\documentclass{article}
\usepackage[utf8]{inputenc}
\usepackage{amssymb}
\usepackage{amsmath}
\usepackage{hyperref}
\usepackage{amsthm}
\usepackage{mathrsfs}
\usepackage{mathtools}
\usepackage{tikz}
\usepackage{bbm}
\usepackage{hyperref}
\usepackage{tikz-cd}
\usepackage{comment}
\usepackage{fullpage}

\usepackage{sectsty}
\usepackage{setspace}

\title{$\pm$-$P$-adic $L$-functions for $\mathrm{GL_{2n}}$}
\author{by Rob Rockwood}
\date{}

\newtheorem{theorem}{Theorem}[section]
\newtheorem{lemma}[theorem]{Lemma}

\newtheorem{proposition}[theorem]{Proposition}
\theoremstyle{definition}

\newtheorem{definition}[theorem]{Definition}
\newtheorem{remark}[theorem]{Remark}
\newtheorem{corollary}[theorem]{Corollary}

\begin{document}

\maketitle

\begin{abstract}
      We generalise Pollack's construction of plus/minus L-functions to certain cuspidal automorphic representations of $\mathrm{GL_{2n}}$ using the $p$-adic $L$-functions constructed in work of Barrera, Dimitrov and Williams \cite{ChrisShalika}. We use these to prove that the complex $L$-functions of such representations vanish at at most finitely many twists. 
\end{abstract}

\section{Introduction}
Let $f = \sum_{n = 0}^{\infty}a_{n}q^{n}$ be a normalized cuspidal newform of weight $k$ and level $N$ with character $\varepsilon$, and let $p$ be a prime such that $p \nmid N$. Let $\alpha$ be a root of the Hecke polynomial $X^{2} - a_{p}X + p^{k - 1}\varepsilon(p)$ which,
 after fixing an isomorphism $\bar{\mathbb{Q}}_{p} \cong \mathbb{C}$, satisfies $r := v_{p}(a_{p}) < k - 1$, where $v_{p}$ is the $p$-adic valuation on $\mathbb{C}_{p}$ normalized so that $v_{p}(p) = 1$. From this data we can construct an order $r$ locally analytic distribution $L_{p}^{(\alpha)}$ on $\mathbb{Z}_{p}^{\times}$ whose values at special characters interpolate the critical values of the complex $L$-function of $f$ and its twists. The arithmetic of $L_{p}^{(\alpha)}$ is well understood in the case that $f$ is ordinary at $p$ i.e. when $r = 0$, but is more mysterious in the non-ordinary case, since the unbounded growth of $L_p^{(\alpha)}$ means that it does not lie in the Iwasawa algebra, and hence cannot be the characteristic element of an Iwasawa module. 
 
 In \cite{pollack} Pollack provides a solution to this problem in the case that $a_{p} = 0$ by constructing bounded distributions $L^{+}_{p}, L^{-}_{p}$ each of which interpolate half the values of the complex $L$-function of $f$ and its twists. Kobayashi \cite{kob} and Lei \cite{lei} have formulated Iwasawa main conjectures using these `$\pm$ $p$-adic $L$-functions', shown them to be equivalent to Kato's main conjecture and proved one inclusion in these conjectures using Kato's Euler system. The converse inclusion has been proved in many cases by Wan \cite{wan}. 
 
 Now let $\Pi$ be a cuspidal automorphic representation of $\mathrm{GL}_{2n}(\mathbb{A}_{\mathbb{Q}})$. Suppose that $\Pi$ is cohomological with respect to some dominant integral weight $\mu$, and that it is the transfer of a globally generic cuspidal automorphic representation of $\mathrm{GSpin}_{2n + 1}(\mathbb{A}_{\mathbb{Q}})$. Let $p$ be a prime at which $\Pi$ is unramified, and let $\alpha_{1}, \ldots, \alpha_{2n}$ be the Satake parameters at $p$. We call a choice of $\alpha = \prod_{i = 1}^{n}\alpha_{j_{i}}$ a $p$-stabilisation of $\Pi$. When a $p$-stabilisation $\alpha$ is \textit{non-critical} and under some further auxiliary technical assumptions Dimitrov, Januszewski and Raghuram \cite{dimitrov} (ordinary case) and Barrera, Dimitrov and Williams \cite{ChrisShalika} construct a locally analytic distribution $L^{(\alpha)}_{p}$ on $\mathbb{Z}_{p}^{\times}$ interpolating the $L$-values of $\Pi$. If we assume $\alpha$ satisfies a small slope condition then this $p$-stabilisation is non-critical, although this is a stronger condition. We show that that there are at most two choices of $\alpha$ satisfying the small slope condition and thus at most two non-critical slope $L^{(\alpha)}_{p}$ can be constructed from a given $\Pi$. Suppose we have two non-critical slope $p$-adic $L$-functions for a given $\Pi$ and suppose the following condition, which we dub the `Pollack condition', holds:
 \begin{equation} \label{pol}
 \textbf{Pollack condition:} \ \alpha_{n} + \alpha_{n + 1} = 0.
 \end{equation}
 See Remark \ref{rem:1} for examples of $\Pi$ satisfying this condition. 
 
 We prove the following theorem:
 \begin{theorem}
 Let $\alpha$ be as above and let $\mathrm{Crit}(\Pi)$ be the set of critical integers for $\Pi$ defined in Definition \ref{def:1}. There exist a pair of distributions $L_{p}^{\pm}$ of order $v(\alpha) - \#\mathrm{Crit}(\Pi)/2$ satisfying
 $$
    L_{p}^{(\alpha)} = \mathrm{log}_{\Pi}^{+}L^{+}_{p} + \mathrm{log}_{\Pi}^{-}L^{-}_{p},
 $$
 where $\mathrm{log}^{\pm}_{\Pi}$ are distributions depending only on $\mathrm{Crit}(\Pi)$ of order $\#\mathrm{Crit}(\Pi)/2$. If the valuation of $\prod_{i = 1}^{n - 1}\alpha_{i}$ is minimal the distributions $L^{\pm}_{p}$ are bounded. These distributions satisfy the following interpolation property for $j \in \mathrm{Crit}(\Pi)$:
$$
     \int_{\mathbb{Z}_{p}^{\times}}x^{j}\theta(x)L_{p}^{+}(x) = (*) \frac{L(\Pi\otimes\theta, j + 1/2)}{\mathrm{log}_{\Pi}^{+}(x^{j}\theta)}
$$
 for $\theta$ a Dirichlet character of conductor an even power of $p$, and
 $$
    \int_{\mathbb{Z}_{p}^{\times}}x^{j}\theta(x)L_{p}^{-}(x) = (*) \frac{L(\Pi\otimes\theta, j + 1/2)}{\mathrm{log}_{\Pi}^{-}(x^{j}\theta)}
 $$
 for $\theta$ a Dirichlet character of conductor an odd power of $p$, where the $(*)$ are non-zero constants.
 \end{theorem}
  \begin{remark}
  The condition that $\prod_{i = 1}^{n - 1}\alpha_{i}$ be minimal is equivalent to the statement that this $p$-stabilisation is $\mathscr{P}$-ordinary where $\mathscr{P} \subset \mathrm{GL}_{2n}$ is the parabolic subgroup given by the partition $2n = (n - 1) + 2 + (n - 1)$.
  \end{remark}
  As an application we prove the following generalisation of the main result of \cite{dimitrov}:
  \begin{theorem}
  In the case that $L_{p}^{\pm}$ are bounded distributions, the purity weight $w$ is even, and $\mathrm{Crit}(\Pi) \neq \{w/2\}$, we have
$$
    L(\Pi \otimes \theta, (w + 1)/2) \neq 0
$$
for all but finitely many characters $\theta$ of $p$-power conductor.
  \end{theorem}
\begin{remark}
The assumption on the purity weight is to ensure that the $L$-value is central.
\end{remark}

\textbf{Relation to other work}: Since this paper first appeared in preprint form, Lei and Ray \cite{leiray} have used the results of this paper to formulate an Iwasawa main conjecture for $\Pi$, relating the signed $p$-adic L-functions of Theorem 1.0.1 to signed Selmer groups. They have also generalised the construction of the signed $p$-adic $L$-functions to allow certain cases with $\alpha_{n} + \alpha_{n + 1} \neq 0$, using the theory of Wach modules.
 
\section{$p$-adic Fourier theory} \label{sec:10}

We lay out the relevant theory of continuous functions on $\mathbb{Z}_{p}^{\times}$. The main reference for this section is \cite{colmez}.

Let $L$ be a complete extension of $\mathbb{Q}_{p}$, denote by $\mathscr{C}(\mathbb{Z}_{p}, L)$ the Banach space of continuous functions on $\mathbb{Z}_{p}$ taking values in $L$ and write $\mathrm{LA}(\mathbb{Z}_{p}, L)$ for the subspace of locally analytic functions.

\begin{definition}
Let $r \in \mathbb{R}_{\geq 0}$. Let $f: \mathbb{Z}_{p} \to L \in \mathscr{C}(\mathbb{Z}_{p}, L)$. We say that $f$ is of \textit{order $r$} if there are functions $f^{(i)}: \mathbb{Z}_{p} \to L$ such that if we define
$$
    \varepsilon_{n}(f) = \mathrm{inf}_{\substack{x \in \mathbb{Z}^{\times}_{p} \\ y \in 1 + p^{n}\mathbb{Z}_{p}}}v_{p}\left(f(x + y) - \sum_{i = 0}^{\lfloor r \rfloor}f^{(i)}(x)x^{i}/i!\right), 
$$
then 
$$
    \varepsilon_{n}(f) - rn \to \infty \ \text{as} \ n \to \infty.
$$
We denote the set of such functions by $  \mathscr{C}^{r}(\mathbb{Z}_{p}, L)$.
\end{definition}
The set $  \mathscr{C}^{r}(\mathbb{Z}_{p}, L)$  is a Banach space with valuation given by 
    $$
        v_{\mathscr{C}^{r}}(f) = \mathrm{inf}\left(\mathrm{inf}_{0 \leq j \leq \lfloor r \rfloor, x \in \mathbb{Z}_{p}}(\frac{f^{(i)}(x)}{i !}), \mathrm{inf}_{x,y \in \mathbb{Z}_{p}}\left(\varepsilon_{n}(f) - rv_{p}(y)\right)\right).
    $$
\begin{definition}
\begin{itemize}
\item Define the space of \textit{locally analytic distributions} on $\mathbb{Z}_{p}$ to be the continuous dual of $\mathrm{LA}(\mathbb{Z}_{p}, L)$, denoted $\mathscr{D}(\mathbb{Z}_{p}, L)$.

\item For $r \in \mathbb{R}_{\geq 0}$ define the subspace $\mathscr{D}^{r}(\mathbb{Z}_{p}, L)$ of $\mathscr{D}(\mathbb{Z}_{p}, L)$ of \textit{order $r$ distributions} to be the continuous dual of $\mathscr{C}^{r}(\mathbb{Z}_{p}, L)$.
The space $\mathscr{D}^{0}(\mathbb{Z}_{p}, L)$ of bounded distributions is often referred to as the space of \textit{measures} on $\mathbb{Z}_{p}$.
\end{itemize}
\end{definition}
 For $\mu \in \mathscr{D}^{r}(\mathbb{Z}_{p}, L), f \in \mathscr{C}^{r}(\mathbb{Z}_{p}, L)$ we write
$$
    \mu(f) =: \int_{\mathbb{Z}_{p}}f(x)\mu(x).
$$
The space $\mathscr{D}(\mathbb{Z}_{p}, L)$ is given the structure of an $L$-algebra  via convolution of distributions: 
$$
    \int_{\mathbb{Z}_{p}}f(x) (\mu*\lambda(x)) := \int_{\mathbb{Z}_{p}}\left(\int_{\mathbb{Z}_{p}}f(x + y)\mu(x)\right)\lambda(y).
$$

These distribution spaces will be our main object of our study. Though rather inscrutable by themselves, they become more amenable to study by identifying them with spaces of power series.

For $x \in \mathbb{C}_{p}, a \in \mathbb{R}$, let $\bar{\mathbb{B}}(x, a) = \{y \in \mathbb{C}_{p}: v_{p}(y - x) \geq a\}$, and $\mathbb{B}(x, a) = \{y \in \mathbb{C}_{p}: v_{p}(y - x) > a\}$. We define 
$$
    \mathscr{R}^{+} = \left\{ f = \sum_{n  = 0}^{\infty}a_{n}(f)X^{n} \in L[[X]] : f \ \text{converges on} \ \mathbb{B}(0,0)\right\}.
$$
Let $\ell(n) = \mathrm{inf}\{m : n < p^{m}\}$, and for $r \in \mathbb{R}_{\geq 0}$ define
$$
    \mathscr{R}^{+}_{r} = \{ f = \sum_{n  = 0}^{\infty}a_{n}X^{n} \in L[[X]] : v_{p}(a_{n}) + r\ell(n) \ \text{is bounded below as} \ n \to \infty \}.
$$

Let $u_{h} = (p^{n}(1 - p))^{-1}$ and let $v_{\bar{\mathbb{B}}(0, u_{h})}$ be the valuation on power series which converge on $\bar{\mathbb{B}}(0, u_{h})$ given by 
$$
   v_{\bar{\mathbb{B}}(0, u_{h})}(f) = \mathrm{inf}_{n}\{v_{p}(a_{n}): f(X) = \sum_{n = 0}^{\infty}a_{n}X^{n}\}.
$$

\begin{lemma} \label{lem:29}
A power series $f \in L[[X]]$ is in $\mathscr{R}^{+}_{r}$ if and only if $ \mathrm{inf}_{h \in \mathbb{Z}_{\geq 0}}(v_{\bar{\mathbb{B}}(0, u_{h})}(f) + rh) \neq -\infty$. Furthermore, the spaces $\mathscr{R}^{+}_{r}$ are Banach spaces when equipped with the valuation
$$
    v_{r}(f) = \mathrm{inf}_{h \in \mathbb{Z}_{\geq 0}}(v_{\bar{\mathbb{B}}(0, u_{h})}(f) + rh).
$$
\end{lemma}
\begin{proof}
\cite[Lemme II.1.1]{colmez}.
\end{proof}
\begin{lemma}
If $f \in \mathscr{R}^{+}_{r}$, $g \in  \mathscr{R}^{+}_{s}$, then $fg \in  \mathscr{R}^{+}_{r + s}$.
\end{lemma}
\begin{proof}
\cite[Corollaire II.1.2]{colmez}.
\end{proof}

 \begin{theorem}
   The map 
 \begin{align*}
     \mathscr{A}: \mathscr{D^{0}}(\mathbb{Z}_{p}, L) &\cong \mathscr{R}^{+} \\
     \mu &\mapsto \int_{\mathbb{Z}_{p}} (1 + X)^{x}\mu(x),
 \end{align*}
 is an isomorphism of $L$-algebras under which the spaces $\mathscr{D}^{r}(\mathbb{Z}_{p}, L)$ and $\mathscr{R}^{+}_{r}$ are identified.
 \end{theorem}
\begin{proof}
\cite[Th\'eor\`eme II.2.2]{colmez} and \cite[Proposition II.3.1]{colmez}.
\end{proof}

We now consider the multiplicative topological group $\mathbb{Z}_{p}^{\times}$. 
We have the well-known isomorphism
$$
    \mathbb{Z}_{p}^{\times} \cong (\mathbb{Z}/p\mathbb{Z})^{\times} \times 1 + p\mathbb{Z}_{p},
$$
the second factor of which is topologically cyclic. Let $\gamma$ be a topological generator of $1 + p\mathbb{Z}_{p}$. Such a choice allows us to write any $x \in 1 + p\mathbb{Z}_{p}$ in the form $x = \gamma^{s}$ for a unique $s \in \mathbb{Z}_{p}$, giving us an isomorphism of topological groups
\begin{align*}
    1 + p\mathbb{Z}_{p} &\cong \mathbb{Z}_{p} \\
    \gamma^{s} &\mapsto s.
\end{align*}
Thus $\mathbb{Z}_{p}^{\times}$ is homeomorphic to $p - 1$ copies of $\mathbb{Z}_{p}$, and we can use the above theory of $\mathbb{Z}_{p}$ in this context, defining $\mathrm{LA}(\mathbb{Z}^{\times}_{p}, L)$, $\mathscr{D}(\mathbb{Z}^{\times}_{p}, L)$ in the obvious way. 
\begin{definition}
Define $\textit{weight space}$ to be the rigid analytic space $\mathcal{W}$ over $\mathbb{Q}_{p}$ representing
$$
    L \mapsto \mathrm{Hom}_{\mathrm{cont}}(\mathbb{Z}_{p}^{\times}, L^{\times}).
$$
\end{definition}

Integrating characters gives a canonical identification
$$
    \mathscr{D}(\mathbb{Z}_{p}^{\times}, \mathbb{Q}_{p}) = H^{0}(\mathcal{W}, \mathcal{O}_{\mathcal{W}})
$$
which commutes with base change in the sense that
$$
 \mathscr{D}(\mathbb{Z}_{p}^{\times}, L) =    \mathscr{D}(\mathbb{Z}_{p}^{\times}, \mathbb{Q}_{p})\hat{\otimes}_{\mathbb{Q}_{p}} L = H^{0}(\mathcal{W}_{L}, \mathcal{O}_{\mathcal{W}_{L}}), 
$$
where $\mathcal{W}_{L} = \mathcal{W} \times_{\mathbb{Q}_{p}}\mathrm{Sp}(L)$.
We identify $\mathcal{W}(\mathbb{C}_{p})$ with the set $\sqcup_{\psi} \mathbb{B}_{\psi}$, where $\mathbb{B}_{\psi}= \mathbb{B}(0,0)$ and the disjoint union runs over characters of $\mathbb{Z}_{p}^{\times}$ which factor through $(\mathbb{Z}/p\mathbb{Z})^{\times}$, for details see $\cite{pollack}$. We can thus identify $\mathscr{D}(\mathbb{Z}_{p}^{\times}, L)$ with functions on $\sqcup_{\psi} \mathbb{B}_{\psi}$ which are described by elements of $\mathscr{R}^{+}$ on each $\mathbb{B}_{\psi}$. Given a distribution $\mu \in \mathscr{D}(\mathbb{Z}_{p}^{\times}, L)$ we write the corresponding locally analytic function on $\mathcal{W}(\mathbb{C}_{p})$ as $\mathscr{M}(\mu)$.

\begin{definition}
For $r \in \mathbb{R}_{\geq 0}$ we define a subspace $ \mathscr{D}^{r}(\mathbb{Z}_{p}^{\times}, L) \subset \mathscr{D}(\mathbb{Z}_{p}^{\times}, L)$ by
$$
\mathscr{D}^{r}(\mathbb{Z}_{p}^{\times}, L) = \{\mu \in \mathscr{D}(\mathbb{Z}_{p}^{\times}, L): \mathscr{M}(\mu)\vert_{\mathbb{B}_{\psi}} \in \mathscr{R}^{+}_{r} \ \text{for all} \ \psi \}
$$ 
for all $\psi$.
\end{definition}

\section{$P$-adic $L$-functions and Pollack's $\pm$ construction} \label{sect3}
\subsection{$P$-stabilisations}
Fix $n \geq 1$ and set $G = \mathrm{GL}_{2n}$. Let $T$ be the diagonal maximal torus of $G$ and $B$ the upper triangular Borel subgroup. Let $\Pi$ be a cuspidal automorphic representation  of $G(\mathbb{A}_{\mathbb{Q}})$ which is cohomological with respect to an integral weight 
$$
    \mu = (\mu_{1}, \ldots, \mu_{2n}) \in \mathbb{Z}^{2n}
$$ 
and suppose that $\Pi$ is the transfer of a globally generic cuspidal automorphic representation of $\mathrm{GSpin}_{2n + 1}(\mathbb{A}_{\mathbb{Q}})$. Such a weight $\mu$ is \textit{dominant}; that is,
$$
\mu_{1} \geq \ldots \geq \mu_{2n}
$$ 
and \textit{pure}: there is an integer $w$ such that 
$$
    \mu_{i} + \mu_{2n + 1 - i} = w,  \ i = 1, \ldots n.
$$
Given a prime $p$ at which $\Pi$ is unramified, define the `Hodge-Tate weights' of $\Pi$ at $p$ to be the integers
$$
    h_{i} = \mu_{i} + 2n - i, \  i = 1, \ldots 2n.
$$

\begin{definition} \label{def:1}
Define a set
$$
    \mathrm{Crit}(\Pi) = \{ j \in \mathbb{Z}: \mu_{n} \geq j \geq \mu_{n + 1}\}.
$$
\end{definition}
\begin{remark}
It is shown in \cite{grobner} that the half integers $j + 1/2$ for $j \in \mathrm{Crit}(\Pi)$ are precisely the critical points of the $L$-function $L(s, \Pi)$ in the sense of Deligne \cite{deligne}. 
\end{remark}
Let $p$ be a prime at which $\Pi$ is unramified. There is an unramified character 
$$
    \lambda_{p}: T(\mathbb{Q}_{p}) \to \mathbb{C}^{\times}
$$
such that $\Pi_{p}$ is isomorphic to the normalised parabolic induction module
$ \mathrm{Ind}_{B(\mathbb{Q}_{p})}^{G(\mathbb{Q}_{p})}(\vert \cdot \vert^{\frac{2n - 1}{2}}\lambda_{p} )
$.
We define the Satake parameters at $p$ to be the values $\alpha_{i} = \lambda_{p, i}(p)$, where $\lambda_{p,i}$ denotes the projection to the $i$th diagonal entry.
After choosing an isomorphism $\bar{\mathbb{Q}}_{p} \cong \mathbb{C}$, we reorder the $\alpha_{i}$ so that they are
ordered with respect to decreasing $p$-adic valuation and such that $\alpha_{i}\alpha_{2n + 1 - i} = \lambda$ for a fixed $\lambda$ with $p$-adic valuation $2n - 1 + w$. That we can do this is a result of the transfer from $\mathrm{GSpin}_{2n + 1}$, see \cite[(6.4)]{asgari}. 

We define the Hodge polygon of $\Pi$ to be the piecewise linear curve joining the following points in $\mathbb{R}^{n}$:
$$
   \left \{(0,0), (j, \sum_{i = 1}
^{j}h_{2n + 1 - i}): j = 1, \ldots,2n \right \}$$
and define the Newton polygon on $\Pi$ at $p$ to be the piecewise linear curve joining the points
$$
   \left \{(0,0), (j, \sum_{i = 1}
^{j}v_{p}(\alpha_{2n + 1 - i})): j = 1, \ldots,2n \right\}.
$$
The following result  is due in this form to Hida \cite{hida}.
\begin{proposition}
The  Newton polygon lies on or above the Hodge polygon and the end points coincide.
\end{proposition}

Let $I = (i_{1}, \ldots, i_{n}) \subset \mathbb{Z}^{n}$ satisfy $1 \leq i_{1} < \ldots < i_{n} \leq 2n$, and set
$$
\alpha_{I} := \alpha_{i_{1}}\cdots \alpha_{i_{n}}.
$$ \begin{definition}
Let $I$ be as above.
\begin{itemize}
\item  We say that the product $\alpha_{I}$ is of \textit{Shalika type} if $I$ contains precisely one element of each pair $\{i, 2n + 1 - i\}$ for $i = 1, \ldots, n$.
\item We say that $\alpha_{I}$ is \textit{$Q$-regular} if it is of Shalika type and if for any other choice of $J \subset \mathbb{Z}^{n}$ satisfying the above properties $a_{J} \neq a_{I}$.  This amounts to choosing a simple Hecke eigenvalue on the parahoric invariants of $\Pi_{p}$, see \cite[\S 3.1]{dimitrov} and \cite[\S2.7]{ChrisShalika}.
\item Set $r_{I} = v_{p}(\alpha_{I}) - \sum_{i = n + 1}^{2n}h_{i}$. We say that $\alpha_{I}$ is \textit{non-critical slope} if it satisfies
$$
   r_{I} < \#\mathrm{Crit}(\Pi).
$$
\end{itemize}
\end{definition}

In \cite{ChrisShalika} the authors construct \footnote{The authors actually construct $p$-adic $L$-functions for the wider class of \textit{non-critical} $p$-stabilisations, but we only work with non-critical slope $p$-stabilisations in this paper.} a locally analytic distribution $L^{(\alpha_{I})}_{p} \in \mathscr{D}(\mathbb{Z}_{p}^{\times}, \mathbb{C}_{p}^{\times})$ with respect to a choice of non-critical slope $Q$-regular `$p$-stabilization data' $\alpha_{I}$. The distribution  $L_{p}^{(\alpha_{I})}$ is of order $r_{I}$ and by \cite{visik} Lemma 2.10 is uniquely defined by the following interpolation property: Let $\theta: \mathbb{Z}^{\times}_{p} \to \bar{\mathbb{Q}}_{p}$ be a finite-order character of conductor $p^{m}$, then for $m \geq 1$ we have
\begin{equation} \label{eq:4}
    \int_{\mathbb{Z}^{\times}_{p}}x^{j}\theta(x) L^{(\alpha_{I})}_{p}(x) = \xi_{\infty,j} \frac{c_{x^{j}\theta}}{\alpha_{I}^{m}}L(\Pi\otimes\theta, j + 1/2), \ j \in \mathrm{Crit}(\Pi),
\end{equation}
where $c_{x^{j}\theta}$ is a constant depending only on $x^{j}\theta$ and the infinite factor $\xi_{\infty, j}$ is the product of a choice of period and a zeta integral at infinity. We call such a $L_{p}^{(\alpha_{I})}$ a `non-critical slope $p$-adic $L$-function'. This condition suffices to prove the following bound on the number of possible non-critical slope $p$-adic $L$-functions attached to $\Pi$.

\begin{theorem} \label{thm:1}
For $\Pi$ as above there are at most two choices of $p$-stabilization $\alpha_{I}$ for which $L^{(\alpha_{I})}_{p}$ is non-critical slope.
\end{theorem}
\begin{proof}
Without loss of generality we may assume that $\mu_{2n} = 0$, forcing $w = \mu_{1}$. The end points of the Newton and Hodge polygons coinciding implies that 
\begin{equation} \label{eq:5}
    v_{p}(\lambda) = h_{i} + h_{2n + 1 - i}, \ i = 1, \ldots, n. \tag{$\dagger$}
\end{equation}
The `non-critical slope' condition for $I = (i_{1}, \ldots, i_{n})$ is equivalent to 
$$
    v_{p}(\alpha_{I}) - \sum_{i = n+ 1}^{2n}h_{i} < h_{n} - h_{n + 1}.
$$

We observe that any $I$ that includes a a 2-tuple of integers the form  $(i, 2n + 1 - i)$ is not non-critical slope. Indeed, we can find an explicit $I$ containing some $(i, 2n + 1 - i)$ with minimal valuation, namely $(n, n + 1, n + 3, \ldots, 2n)$, amongst all $I$ containing some $(i, 2n + 1 - i)$. For such an $I$ we have
\begin{align*}
    v_{p}(\alpha_{I}) - (h_{n + 1} + h_{n + 2} \ldots + h_{2n}) &\geq h_{n} + h_{n + 1} + h_{n + 3} + \dots + h_{2n} - (h_{n + 1} +  \ldots + h_{2n})\\
    &= h_{n} - h_{n + 2} \\
    &> h_{n} - h_{n + 1},
\end{align*}
where the first inequality is a consequence of the Newton polygon lying above the Hodge polygon and \eqref{eq:5}, and the strict inequality is due to dominance. Thus any $I$ containing a pair of integers $(i, j)$ with $i < j$ and $j \leq 2n + 1 - i$ cannot be non-critical slope, since any such $I$ has greater valuation than $(n, n + 1, n + 3, \ldots, 2n)$. 
This leaves us with two choices of potential non-critical slope $n$-tuples:
$$
    I_{n + 1} = (n + 1, n + 2, \ldots, 2n),
$$
and
$$
    I_{n} = (n, n + 2, \ldots, 2n).
$$
\end{proof}
In light of Theorem \ref{thm:1} it is clear that the only two choices of $p$-stabilization data which can give a non-critical slope distribution are 
\begin{align*}
      \alpha = \alpha_{n + 1}\alpha_{n + 2}\ldots\alpha_{2n}, \ \beta = \alpha_{n}\alpha_{n + 2}\ldots\alpha_{2n}.
\end{align*}

\begin{corollary}
A non-critical slope $p$-stabilization is of Shalika type.
\end{corollary}

\subsection{Pollack $\pm$-$L$-functions}
Let $\Pi$ be as in the previous section. 
\begin{definition}
We say that $\Pi$ satisfies the `Pollack condition' if
$$
    \alpha_{n} + \alpha_{n + 1} = 0.
$$
\end{definition}
\begin{remark} \label{rem:1}
If one considers a quartic CM-field $E$ such that there exists a rational prime $p$ which splits in the totally real subfield $F$ of $E$ into two prime ideals, one of which splits in $E$ and the other of which is inert, then one can construct a Hecke character over $E$ which, by automorphic induction to $GL_{4}/\mathbb{Q}$ via $\mathrm{GL}_{2}/F$, induces a representation of $\mathrm{GL}_{4}$ satisfying the Pollack condition.
\end{remark}
As in the previous section, set 
$$
    \alpha = \alpha_{n + 1}\alpha_{n + 2}\ldots\alpha_{2n}, \ \beta = \alpha_{n}\alpha_{n + 2}\ldots\alpha_{2n},
$$
and let $r = v_{p}(\alpha) - \sum_{i = n + 1}^{2n}h_{i} = v_{p}(\beta) - \sum_{i = n + 1}^{2n}h_{i}$. The Pollack condition forces
$$
    r \geq \#\mathrm{Crit}(\Pi)/2
$$
since
\begin{align*}
   r=  v_{p}(\alpha) - \sum_{i = n + 1}^{2n}h_{i} &\geq v_{p}(\alpha_{n + 1}) - h_{n + 1 } \\
    &= \frac{h_{n} + h_{n + 1}}{2} - h_{n + 1} \\
    &= \frac{h_{n} - h_{n + 1}}{2} \\
    &= \#\mathrm{Crit}(\Pi)/2,
\end{align*}
where the first inequality comes from Newton-above-Hodge, and the lower bound given is tight, with the bound being achieved when $\alpha_{n + 2}\cdots\alpha_{2n}$ is of minimal slope.
We assume that 
$$
    r  < \#\mathrm{Crit}(\Pi)
$$
so that we can construct precisely two non-critical slope $p$-adic  $L$-functions $L^{(\alpha)}_{p}, L_{p}^{(\beta)} \in \mathscr{D}^{r}(\mathbb{Z}_{p}^{\times}, \mathbb{C}_{p})$.
\begin{remark}
Unlike in the case of $\mathrm{GL}_{2}$, for $n \geq 1$ the non-critical slope condition for $\alpha, \beta$ is not implied by the Pollack condition. Indeed, one can construct a cuspidal automorphic representation $\Pi$ of $\mathrm{GL}_{4}(\mathbb{A}_{\mathbb{Q}})$ satisfying the Pollack condition at a prime $p$ for which $v_{p}(\alpha_{i}) = v_{p}(\alpha_{j})$ for $1 \leq i,j \leq 4$. The value $r$ is then the same for any choice of $p$-stabilization, so there are either six non-critical slope $p$-adic $L$-functions or there are none. But Theorem \ref{thm:1} says there can be at most two choices of non-critical slope $p$-stabilization.
\end{remark}
Following Pollack, we define 
$$
    G^{\pm} = \frac{L^{(\alpha)}_{p} \pm L^{(\beta)}_{p}}{2}, 
$$
so that
\begin{align*}
    L^{(\alpha)}_{p} &= G^{+} + G^{-} \\
    L^{(\beta)}_{p} &= G^{+} - G^{-}.
\end{align*}
 We note that in the case of $L^{(\beta)}_{p}$, \eqref{eq:4} looks like
$$
     \int_{\mathbb{Z}_{p}^{\times}}x^{j}\theta(x) L^{(\beta)}_{p}(x) = (-1)^{m}\xi_{\infty, j}\frac{c_{x^{j}\theta}}{\alpha^{m}}L(\Pi\otimes\theta, j + 1/2), \ j \in \mathrm{Crit}(\Pi),
$$
from which it follows that
\begin{align*}
    \int_{\mathbb{Z}_{p}^{\times}}x^{j}\theta(x) G^{+}(x) &= 0, \ \text{if the conductor of $\theta$ is $p^{m}$, $m$ even} \\
    \int_{\mathbb{Z}_{p}^{\times}}x^{j}\theta(x) G^{-}(x) &= 0, \ \text{if the conductor of $\theta$ is $p^{m}$, $m$ odd}. 
\end{align*}
Equivalently, if $\zeta_{p^{m}}$ is any $p^{m}$th root of unity,
\begin{align*}
    \mathscr{M}(G^{+})(\gamma^{j}\zeta_{p^{m}} - 1) = 0 \ &\text{for $m$ even} \\
     \mathscr{M}(G^{-})(\gamma^{j}\zeta_{p^{m}} - 1) = 0 \ &\text{for $m$ odd}
\end{align*}
on each of the connected components\footnote{We are referring to the connected components as a rigid space as opposed to those of the topology on $\mathbb{C}_{p}$ induced by $v_{p}$.}  of $\mathcal{W}(\mathbb{C}_{p})$ (which we recall we are identifying with $p- 1$ copies of $\mathbb{B}(0,0)$).

For any $j \in \mathbb{Z}$, Pollack defines the following power series 

\begin{align*}
    \mathrm{log}_{p,j}^{+}(X) &:= \frac{1}{p}\prod_{m = 1}^{\infty}\frac{\Phi_{2m}(\gamma^{-j}(1 + X))}{p} \\
    \mathrm{log}_{p,j}^{-}(X) &:= \frac{1}{p}\prod_{m}^{\infty}\frac{\Phi_{2m - 1}(\gamma^{-j}(1 + X))}{p},
\end{align*}
in $\mathbb{Q}_{p}[[X]]$, where $\Phi_{m}$ is the $p^{m}$th cyclotomic polynomial.

\begin{lemma} \label{lem:30}
The power series $\mathrm{log}^{+}_{p, j}(X)$ (resp. $\mathrm{log}^{-}_{p, j}( X)$) is contained in $\mathscr{R}_{1/2}^{+}$ and vanishes at precisely the points $\gamma^{j}\zeta_{p^{m}} - 1$ for every $p^{m}$th root of unity $\zeta_{p^{m}}$ with $m$ even (resp. odd).
\end{lemma}
\begin{proof}
The statements in the lemma are proved in \cite[Lemma 4.1]{pollack} and \cite[Lemma 4.5]{pollack}. We reprove that $\mathrm{log}^{+}_{p,j}$ is contained in $\mathscr{R}^{+}_{1/2}$ using the setup of Section \ref{sec:10}, the result for $\mathrm{log}^{-}_{p,j}$ being similar. 

An analysis of the Newton copolygon of the Eisenstein polynomial $\Phi_{n}$ gives us that 
$$
    v_{\bar{\mathbb{B}}(0, u_{h})}(\Phi_{n}((\gamma^{-j}(1 + X))/p) = \begin{cases}
    0  &\text{if} \ h \leq n - 1 \\
    p^{n - h - 1} - 1 &\text{otherwise},
     \end{cases}
$$
and thus 
\begin{align*}
    v_{\bar{\mathbb{B}}(0, u_{h})}(\mathrm{log}^{+}_{p,j}) &= \sum^{ \frac{h + 1}{2}}_{m = 1}\left (p^{2m - h - 1} - 1 \right )\\
    &= \frac{p^{-(h + 1)} - 1}{1 - p^{2}} - \frac{1}{2} - \frac{h}{2},
\end{align*}
whence 
$$
    \mathrm{inf}_{h} \left ( v_{\bar{\mathbb{B}}(0, u_{h})}(\mathrm{log}^{+}_{p, j}) + \frac{h}{2} \right) = \frac{1}{p^{2} - 1} - \frac{1}{2} < \infty,
$$
so $\mathrm{log}^{+}_{p,j}(X) \in \mathscr{R}^{+}_{1/2}$ by Lemma \ref{lem:29}. 
\end{proof}

We define 
$$
    \mathrm{log}^{\pm}_{\Pi}(X) = \prod_{j \in \#\mathrm{Crit}(\Pi)}\mathrm{log}^{\pm}_{p,j}(X) \in \mathscr{R}^{+}_{\#\mathrm{Crit}(\Pi)/2}.
$$
By abuse of notation we will write $\mathrm{log}^{\pm}_{\Pi}(X)$ for the function on $\mathcal{W}(\mathbb{C}_{p})$ given by $\mathrm{log}^{\pm}_{\Pi}(X)$ on each connected component of $\mathcal{W}(\mathbb{C}_{p})$. 

\begin{lemma} \label{lem:28}
We have
$$
   \limsup_{h}\left(  v_{\bar{\mathbb{B}}(0, u_{h})}(\mathrm{log}^{\pm}_{\Pi}) + \frac{\#\mathrm{Crit}(\Pi)}{2}h \right) < \infty
$$
\end{lemma}
\begin{proof}
It follows from the proof of Lemma \ref{lem:30} and the multiplicativity of $ v_{\bar{\mathbb{B}}(0, u_{h})}$ that 
$$
    v_{\bar{\mathbb{B}}(0, u_{h})}(\mathrm{log}^{+}_{\Pi}) + \frac{\#\mathrm{Crit}(\Pi)}{2}h = \#\mathrm{Crit}(\Pi)\left (  \frac{p^{-(h + 1)} - 1}{1 - p^{2}} - \frac{1}{2} \right).
$$
The right side converges as $h \to \infty$ so the $\mathrm{\limsup}$ is finite. A similar argument works for $\mathrm{log}^{-}_{\Pi}$.
\end{proof}
It follows from the above discussion and \cite[4.7]{lazard} that  $\mathrm{log}^{\pm}_{\Pi}(X)$ divides $\mathscr{M}(G^{\pm})$, and we define $\pm$-$L$-functions $L^{\pm}_{p}(X)$ to be the elements of $\mathcal{O}(\mathcal{W})$ satisfying
$$
    \mathscr{M}(G^{\pm}) = \mathrm{log}_{\Pi}^{\pm}(X) \cdot L^{\pm}_{p}(X).
$$
We write $L^{\pm}_{p}$ for the distribution $\mathscr{M}^{-1}(L_{p}^{\pm}(X))$.
\begin{proposition}
$$
    L^{\pm}_{p} \in \mathscr{D}^{r - \#\mathrm{Crit}(\Pi)/2}(\mathbb{Z}_{p}^{\times}, \mathbb{C}_{p}).
$$
\end{proposition}
\begin{proof}
We note that 
$$
    \liminf_{h} \left(-v_{\bar{\mathbb{B}}(0, u_{h})}(\mathrm{log}^{\pm}_{\Pi}) - \tfrac{\#\mathrm{Crit}(\Pi)}{2}h \right) = -\limsup \left( v_{\bar{\mathbb{B}}(0, u_{h})}(\mathrm{log}^{\pm}_{\Pi}) + \tfrac{\#\mathrm{Crit}(\Pi)}{2}h \right) > -\infty.
$$
By the additivity of $v_{\bar{\mathbb{B}}(0, u_{h})}$ (\cite[Proposition I.4.2]{colmez}) we have
$$
    v_{\bar{\mathbb{B}}(0, u_{h})}(L^{\pm}_{p}) + (r - \tfrac{\#\mathrm{Crit}(\Pi)}{2})h =  v_{\bar{\mathbb{B}}(0, u_{h})}(G^{\pm}) + rh -  v_{\bar{\mathbb{B}}(0, u_{h})}(\mathrm{log}^{\pm}_{\Pi}) - \tfrac{\#\mathrm{Crit}(\Pi)}{2}h),
$$
whence 
$$
    \liminf  v_{\bar{\mathbb{B}}(0, u_{h})}(L^{\pm}_{p}) + (r - \tfrac{\#\mathrm{Crit}(\Pi)}{2})h > -\infty 
$$
and so by Lemma \ref{lem:29} we are done. 
\end{proof}
In particular, in the minimal slope case $r = \frac{\#\mathrm{Crit}(\Pi)}{2}$, we get two bounded distributions.

\subsection{Non-vanishing of twists}
We use $L_{p}^{\pm}$ to show non-vanishing of the complex $L$-function of $\Pi$ at the central value, extending work of Dimitrov, Januszewski, Raghuram \cite{dimitrov} to a non-ordinary setting. 
\begin{proposition}\label{prop:3}
In the case that $\mathrm{Crit}(\Pi) \neq \{ w/2\},$ we have
$$
    L_{p}^{\pm} \neq 0.
$$
\end{proposition}
 \begin{proof}
 We consider $L_{p}^{+}$, the case of $L_{p}^{-}$ being essentially identical. Note that a character $\theta$ of conductor $p^{m}$ is uniquely determined by a choice of primitive $p^{m}$th root of unity $\zeta_{p^{m}}$ which gives us the identification 
 $$
    \mathrm{log}^{+}_{\Pi}(\gamma^{j}\zeta_{p^{m}} - 1) = \int_{\mathbb{Z}_{p}^{*}}x^{j}\theta(x)\mathrm{log}^{+}_{\Pi}(x),
 $$
 where on the right hands side we use the description of $\mathrm{log}^{+}_{\Pi}$ as a distribution and on the left hand side we use the description as a power series. It follows from Lemma \ref{lem:30} that $\int_{\mathbb{Z}_{p}^{*}}x^{j}\theta(x)\mathrm{log}^{+}_{\Pi}(x) \neq 0$ when $m$ is odd.
 Thus for characters $\theta$ of even $p$-power conductor we have the interpolation property
 $$
    \int_{\mathbb{Z}_{p}^{\times}}x^{j}\theta(x)L_{p}^{+}(x) \sim \frac{L(\Pi\otimes\theta, j + 1/2)}{\int_{\mathbb{Z}_{p}^{\times}}x^{j}\theta(x)\mathrm{log}_{\Pi}^{+}(x)}, \ j \in \mathrm{Crit}(\Pi),
 $$
 where $\sim$ is used here to mean `up to non-zero constant'.
 By Jacquet-Shalika \cite[1.3]{jacquetShalika} we have
 $$
    L(s, \Pi \circ \theta) \neq 0 \ \text{for} \ \mathfrak{Re}(s) \geq w/2 + 1
$$
 for finite order characters $\theta$, and by applying the functional equation we get non-vanishing for $\mathfrak{Re}(s) \leq w/2$. Since $\mathrm{Crit}(\Pi)$ contains an integer $m$ not equal to $w/2$, the above discussion gives us
 $$
    L(\Pi, m + 1/2) \neq 0,
 $$
 and thus $L^{+}_{p} \neq 0$.
 \end{proof}
 
 \begin{remark} \label{rmk:1}
Proposition \ref{prop:3} actually proves the stronger result that the power series $\mathscr{M}(L_{p}^{\pm})\vert_{\mathbb{B}_{\psi}}$ is non-zero for each choice of $\psi$.  
 \end{remark}
We can turn this back on itself and use $L^{\pm}_{p}$ to say something about nonvanishing of $L(\Pi \otimes \theta, (\omega + 1)/2)$ in the case when $L_{p}^{\pm} \in \mathscr{D}^{0}(\mathbb{Z}_{p}^{\times}, \mathbb{C}_{p})$.

\begin{theorem}
In the case that $L_{p}^{\pm}$ are bounded distributions, $w$ is even, and $\mathrm{Crit}(\Pi) \neq \{w/2\}$, we have
$$
    L(\Pi \otimes \theta, (w + 1)/2) \neq 0
$$
for all but finitely many characters $\theta$ of $p$-power conductor.
\end{theorem}
\begin{proof}
For any character $\psi$ of $(\mathbb{Z}/p\mathbb{Z})^{\times}$ we can write $\mathscr{M}(L^{\pm}_{p})\vert_{\mathbb{B}_{\psi}} = L^{\pm}_{p}(\psi, T) \in \mathcal{O}_{L}[[T]] \otimes_{\mathcal{O}_{L}} L$ for some finite extension $L/\mathbb{Q}_{p}$. This power series is non-zero by Proposition \ref{prop:3} and Remark \ref{rmk:1}, and so Weierstrass preparation tells us that each $L^{\pm}(\psi, T)$, and thus $L_{p}^{\pm}$, has only finitely many zeroes. Given any character $\theta$ of $p$-power conductor, we have
 $$
  \int_{\mathbb{Z}_{p}^{\times}} x^{w/2}\theta(x) L_{p}^{?}(x) \sim L(\Pi\otimes \theta, (w + 1)/2),
 $$
where $? \in \{+, -\}$ depends on the conductor of $\theta$ in the usual way.
Thus, for all but finitely many $\theta$, we have
$$
    L(\Pi \otimes \theta, (w + 1)/2) \neq 0.
$$
\end{proof}

\bibliography{PlusMinusLfunc}
\bibliographystyle{spmpsci}

\end{document}